\documentclass{amsart}
\numberwithin{equation}{section}

\usepackage[english,german]{babel}
\usepackage{amsmath}
\usepackage{amsfonts}
\usepackage{amssymb}
\usepackage{amsthm}
\usepackage{amsopn}
\usepackage{amscd}
\usepackage{amsxtra}

\usepackage{mdwlist}
\usepackage{subfigure}
\usepackage{multicol}
\usepackage{pstricks}
\usepackage[utf8]{inputenc}
\usepackage{courier}
\usepackage{url}
\usepackage{xypic}
\usepackage[all,cmtip]{xy}

\usepackage[color,final]{showkeys}
\definecolor{labelkey}{rgb}{1,0,1}
\definecolor{refkey}{gray}{1}

\usepackage{enumitem}
\setenumerate{label=(\roman*)}

\usepackage{listings}

\newtheorem{thm}{Theorem}[section]
\newtheorem{lemma}[thm]{Lemma}
\newtheorem{prop}[thm]{Proposition}

\newtheorem{cor}[thm]{Corollary}

\theoremstyle{definition}
\newtheorem{defn}[thm]{Definition}
\newtheorem{notation}[thm]{Notation}

\theoremstyle{remark}
\newtheorem{example}[thm]{Example}

\numberwithin{equation}{section}

 \newcommand{\ZZ}{\mathbb{Z}}
 \newcommand{\QQ}{\mathbb{Q}}
 \newcommand{\NN}{\mathbb{N}}
\newcommand{\RR}{\mathbb{R}}
\newcommand{\CC}{\mathbb{C}}
\newcommand{\PP}{\mathbb{P}}

 \newcommand{\cB}{\mathcal{B}}

\newcommand{\cG}{\mathcal{G}}

\newcommand{\cN}{\mathcal{N}}
\newcommand{\cO}{\mathcal{O}}

 \newcommand{\D}{\Delta}
\newcommand{\s}{\sigma}
\renewcommand{\S}{\Sigma}

\newcommand{\M}{\overline{M}}
\newcommand{\LM}{\overline{L}}

\newcommand{\Pic}{\mathrm{Pic}}

\newcommand{\cox}{\mathrm{Cox}}

\newcommand{\onetn}{\{1, \ldots, n\}}

\newcommand{\Star}{\mathrm{Star}}

\newcommand{\orthb}{\QQ_{\geq 0}^{|\D|}}
\newcommand{\orthNb}{\NN^{|\D|}}
\newcommand{\qbnd}{\QQ^{|\D|}}

\newcommand{\cl}{\mathrm{cl}}

\newgray{gray1}{.65}
\newgray{gray2}{.75}
\newgray{gray3}{.85}

\newcommand{\fanptwo}{%

		\psset{unit=1.9cm}
	\pspicture(-1.5,-1)(1.5,1.3)
	\psline{->}(0, 1)
	\psline{->}(1, 0)
	\psline{->}(-1, -1)

\rput(-1,-0.8){$\rho_{23}$}
\rput(1,-.15){$\rho_{13}$}
\rput(0.15,1.0){$\rho_{12}$}
	\endpspicture

}

\newcommand{\fanlthree}{%

		\psset{unit=1.9cm}
	\pspicture(-1.5,-1)(1.5,1.3)
	\psline{->}(0, 1)
	\psline{->}(1, 0)
	\psline{->}(-1, -1)
	\psline{->}(1, 1)
	\psline{->}(0, -1)
	\psline{->}(-1,0)

\rput(-1,-0.8){$\rho_{23}$}
\rput(1,-.15){$\rho_{13}$}
\rput(0.15,1.0){$\rho_{12}$}
\rput(1.0,0.8){$\rho_1$}
\rput(-1, -0.1){$\rho_2$}
\rput(0.15,-1){$\rho_3$}
	\endpspicture

}

\author{Paul Larsen}

\title[Permutohedral spaces and the Cox ring of $\overline{M}_{0,\lowercase{n}}$]{Permutohedral spaces and the Cox ring of the moduli space of stable pointed rational curves}
\date{\today}

\address{Humboldt-Universit\"at zu Berlin, Institut f\"ur Mathematik, 10099 Berlin, Germany}
\email{larsen@mathematik.hu-berlin.de}

\numberwithin{equation}{section}

\begin{document}
\selectlanguage{english}
\begin{abstract}
We study the Cox ring of the moduli space of stable pointed rational curves, $\M_{0,n}$, via the closely related permutohedral (or Losev-Manin) spaces $\LM_{n-2}$. Our main result establishes $\binom{n}{2}$ polynomial subrings of $\cox(\M_{0,n})$, thus giving collections of boundary variables that intersect the ideal of relations of $\cox(\M_{0,n})$ trivially. As applications, we give a combinatorial way to partially solve the Riemann-Roch problem for $\M_{0,n}$, and we show that all relations in degrees of $\cox(\M_{0,6})$ arising from certain pull-backs from projective spaces are generated by the Pl\"ucker relations.
\end{abstract}
\maketitle

\tableofcontents

\section{Introduction}
\label{secCox:introduction}
The moduli space of stable pointed rational curves, $\M_{0,n}$, serves as a meeting point for a wide range of mathematics; its study links to higher genus moduli spaces of curves, Gromov-Witten theory and mathematical physics, and phylogenetics. For example, a central open question in the birational geometry of moduli spaces of curves for all genera, the F-conjecture, could be settled by the genus zero case (\cite{MR1887636}).

Recently, Cox rings have provided a useful tool for birational geometry. Introduced in \cite{MR1299003} as a generalization of the homogeneous coordinate ring of projective space to toric varieties, Hu and Keel in \cite{MR1786494} generalized Cox's construction to a broad class of projective varieties (see Definition \ref{def:coxProj}), and proved far-reaching implications for the birational geometry of the variety when its Cox ring is finitely generated. They showed that the projective variety can be recovered as a quotient of the spectrum of its Cox ring by the action of its Picard torus, and that the cones of pseudoffective, nef, and moving divisor classes are polyhedral, admitting chamber decompositions that determine all maps from $X$ to other projective varieties. Hu and Keel dubbed these varieties \emph{Mori dream spaces}, since they are in a sense ideally suited for the minimal model program (for precise statements and details, see \cite{MR1786494}).

Proving finite-generation of Cox rings has therefore garnered significant attention (e.g. \cite{MR2029863}, \cite{mukai}, \cite{MR2278756}, \cite{MR2491903},\cite{MR2601039}). In order to extract information about birational geometry, however, the ideal of relations among generators must also be determined. For $\M_{0,n}$, generators are known only for $n \leq 6$ (\cite{MR2029863},\cite{MR2491903}), and the ideal of relations only for $n\leq 5$ (\cite{MR2029863}).

This paper studies the ideal of relations of the Cox ring of $\M_{0,n}$ as it relates to toric varieties called \emph{permutohedral} or \emph{Losev-Manin} spaces. The main result of this paper is to show that the (polynomial) Cox rings of the permutohedral spaces $\LM_{n-2}$ inject into $\cox(\M_{0,n})$ for all $n$ in a particularly nice way. More specifically, we define a natural identification between the torus-invariant prime divisors of $\LM_{n-2}$ and boundary divisors of $\M_{0,n}$ (see Section \ref{secCox:background} for definitions). Then letting $\{[H], [E_{J}]$: $J \subsetneq \{1, \ldots, n-1\}\}$ denote the Kapranov basis $\Pic(\M_{0,n})$ (see Section \ref{secCox:M0nLnm2}), our main result is
\begin{thm}
\label{thm:main}
The Kapranov morphism $f: \M_{0,n}  \to \LM_{n-2}$ induces a $\Pic$-graded isomorphism
\begin{equation*}
f^*: \cox(\LM_{n-2})  \stackrel{\cong}{\longrightarrow} \bigoplus_{(m_h, m_J) \in \ZZ^{2^{n-2} - n +1}} H^0(\M_{0,n}, m_h H + \sum_{J \subsetneq \{1, \ldots, n-2\}} m_J E_J),
\end{equation*}
defined in terms of boundary variables by $x_{\Delta'_{J \cup \{n\}}} \mapsto x_{\Delta_{J \cup \{n\}}}$.
\end{thm}
Since for each $n$ there correspond $\binom{n}{2}$ different permutohedral spaces according to a choice of poles (see Section \ref{secCox:M0nLnm2}), we obtain $\binom{n}{2}$ polynomial subrings of $\cox(\M_{0,n})$. An immediate corollary is that the corresponding subrings of the ring of generators meet the ideal of relations trivially. In addition, Theorem \ref{thm:main} provides a partial solution to the Riemann-Roch problem for $\M_{0,n}$: for any divisor $D$ with a non-trivial section in the image of $f^*$ above, the dimension of global sections $h^0(\M_{0,n}, D)$ equals that of its preimage in $\LM_{n-2}$, which can be calculated by counting lattice points in a corresponding polytope.

Theorem \ref{thm:main} follows from a detailed study of Kapranov's blow-up construction of $\M_{0,n}$. The key technical result states that the pull-back and proper-transforms of the centers being blown up coincide, even for $n \geq 6$ when the blow-up centers are not all disjoint (see Proposition \ref{prop:M0nPBvsPT}).

As an application, in Section \ref{secCox:CoxFull} we exhibit degrees of $\cox(\M_{0,n})$ that always contain relations and show that, for $n = 6$, these relations are generated by Pl\"ucker relations. The proof follows similar lines to those used in \cite{MR2529093} to prove the Batyrev-Popov conjecture that the Cox rings of del Pezzo surfaces (including $\M_{0,5}$) are quadratic algebras. Theorem \ref{thm:main} provides the bridge between del Pezzo surfaces and $\M_{0,n}$. We hope that this analogy can be extended to find a presentation of the Cox ring of $\M_{0,6}$.

The remainder of the paper is organized as follows. In Section \ref{secCox:background}, we introduce the spaces $\M_{0,n}$ and $\LM_{n-2}$, present basic facts and notation from toric geometry, and then define the Cox ring of a projective variety. Section \ref{secCox:M0nLnm2} contains the proof of Theorem \ref{thm:main}, obtained by studying the blow-up constructions of $\LM_{n-2}$ and $\M_{0,n}$, with attention focused on how divisor classes in $\LM_{n-2}$ pull back to $\M_{0,n}$. Of particular use is the language of `clean intersections.' Originally defined by Bott in \cite{MR0090730} in the context of differential geometry, we use the more algebraic formulation of \cite{MR2595553}. Section \ref{secCox:CoxFull} then contains the application to relations in $\cox(\M_{0,6})$. 

\textbf{Acknowledgements:} I would first like to thank my supervisor, Gavril Farkas, for his many ideas and explanations. It is a pleasure to thank Diane Maclagan, whose input led to numerous improvements. I would also like to thank Nathan Ilten, who provided significant help in the first phases of this project, Martha Bernal, Antonio Laface, and Hendrik S\"u\ss\, for commenting on an earlier version of this paper, and Mark Blume for pointing out the connection to Losev-Manin moduli spaces.

\section{Background and notation}
\label{secCox:background}

In this section, we introduce necessary definitions and concepts from moduli spaces of curves, toric geometry, and Cox rings.
First we give a brief account of $\M_{0,n}$ and the permutohedral space $\LM_{n-2}$, referring to 
\cite{MR1034665}, \cite{MR702953}, \cite{MR2262630}, \cite{larsenThesis}, respectively \cite{MR1237834}, \cite{MR1786500}, \cite{blume1} for more details. All varieties discussed below are defined over the complex numbers.

An element of $\M_{0,n}$ is an isomorphism class of a tree of projective lines with $n$ distinct marked points distributed over the smooth locus such that each irreducible component contains at least three special (that is, marked or singular) points. Of particular interest are the \emph{boundary divisors} in $\M_{0,n}$, which are irreducible loci whose general element has a single node (equivalently, two components). If $J \subseteq \{1, \ldots, n\}$ denotes the marked points on one component of the general element, then we denote the resulting boundary divisor by $\Delta_{J, J^c}$, or $\Delta_{J}$. Intuitively, we can consider $\Delta_J$ as the locus of elements obtainable by allowing the special points of the general element to vary. When two special points collide, the result is a new rational component glued at the `collision' point and containing the two (distinct) colliding points.

Permutohedral spaces were first studied in the context of $\M_{0,n}$ in \cite{MR1237834}, where Kapranov realized $\M_{0,n}$ as an iterated blow-up of $\PP^{n-3}$, with the permutohedral space $\LM_{n-2}$ functioning as the final toric variety occurring in his ordering of the blow-ups (see Section  \ref{secCox:M0nLnm2}). The name for these varieties derives from their defining polytopes, which can be realized as the convex hull of the orbit of the point $(1,2, \ldots, n-2) \in \RR^{n-2}$ under the action of the symmetric group on $n-2$ letters, i.e. as a permutohedron. It was realized in \cite{MR1786500} that these spaces admit a modular interpretation, which we now describe.

An element of the permutohedral space $\LM_{n-2}$ is an isomorphism class of a chain of projective lines with two \emph{poles}, one at each of the end components, and $n-2$ marked points distributed along the smooth locus of the chain. These marked points need not be distinct from one another, but they cannot coincide with the poles. Again, each irreducible component is required to contain at least three special points (marked, singular or pole), or, equivalently, each component must contain at least one marked point.


A key feature of permutohedral spaces is that they are toric varieties described as a composition of blow-ups of projective space along torus-invariant linear centers. It will be useful in the proof of Theorem \ref{thm:main} to describe these blow-ups explicitly in terms of fans and cones, so we now turn attention to toric blow-ups. 

Our sources for toric geometry are \cite{cls} and \cite{MR1234037}. Let $N$ be a free abelian group of rank $d$ with dual lattice $M$, and let $X_\Sigma$ be a toric variety of dimension $d$ with fan $\S \subseteq N \otimes_\ZZ \RR = N_\RR$ and torus $T_N$. We denote the set of $k$-dimensional cones of $\S$ by $\S(k)$. A central fact for what follows is the \emph{Orbit-Cone correspondence} (see e.g. Theorem 3.2.6 of \cite{cls}): there exists a bijective correspondence between cones $\s \in \S$ and $T_N$-orbits $O(\sigma)$ such that if $\s \in \S(k)$, then $\dim(O(\s)) = d-k$. 
\noindent Of particular importance are the closures of these orbits. For $\s \in \S(k)$, let $N_\s$ be the sublattice of $N$ generated by points of $\s \cap N$. Then for $\s \in \S(k)$, the orbit closure $V(\s) = \overline{O(\s)}$ is the toric variety $X_{\Star(\s)}$, where
$\Star(\s) = \{ \overline{\tau}: \s \textrm{ is a face of } \tau\}$,
and $\overline{\tau}$ is the image of $\tau$ under the projection map $N_\RR \to (N  / N_\s)_\RR$ (see Proposition 3.2.7 of \cite{cls}).
\begin{defn}
\label{defnCox:strataToric}
The \emph{codimension-k strata} of the toric variety $X_\S$ are the subvarieties $V(\s)$, where $\s \in \S(k)$. The codimension-one strata of $X_\S$ are called \emph{boundary divisors}, and will be denoted $D_{\rho} = V(\rho)$, where $\rho \in \S(1)$.
\end{defn}
The reason for labeling these divisors as `boundary' is that same as for $\M_{0,n}$: for $\M_{0,n}$, the union of the boundary divisors equals the complement $\M_{0,n} \setminus M_{0,n}$, while in the toric case, the union of the boundary divisors is $X_\S \setminus T_N$. Note further that the strata of the toric variety $X_\S$ share two of the nice properties of the strata of $\M_{0,n}$: the union of the codimension-one strata is a normal crossing divisor, and each codimension-$k$ stratum is a complete intersection of $k$ boundary divisors (see \cite{MR1234037}, Section 5.1).

We next describe the toric interpretation of a blow-up along a torus-invariant center (see Proposition 3.3.15 of \cite{cls}). Let $\S \subseteq N$ be a $d$-dimensional fan, and let $\sigma = \langle u_1, \ldots, u_d \rangle_{\geq 0}$ be a smooth cone (that is, $u_1, \ldots, u_d$ form a $\ZZ$-basis for the lattice $N$). To construct the blow-up of $X_\S$ along $V(\s)$, let $u = u_1 + \ldots + u_d$, and define $\s'$ to be the set of all cones generated by subsets of $\{u, u_1, \ldots, u_d \}$ that do not contain $\{u_1, \ldots, u_d\}$. Then the fan of the blow-up $Bl_{V(\s)}(X_\S)$ is 
\begin{equation}
\S' = (\S \setminus \s) \cup \s'. \nonumber
\end{equation}
The exceptional divisor of the blow-up is $D_u$ (taken in $X_{\S'}$), and the proper transform of the divisor $D_{u_i}$ in $X_{\S}$, $i=1, \ldots, d$, is $D_{u_i} - D_u$ in $X_{\S'}$ (for $\rho$ not a face of $\s$, the proper transform leaves $D_\rho$ unchanged). 

The Kapranov blow-up constructions described in the next section involve iterated blow-ups along linear subspaces of projective space. To harmonize notation for the linear subspaces and the resulting exceptional divisors (and their proper transforms), we use the following conventions.

\begin{notation}[Linear subspaces, blow-ups, and proper transforms]
\label{notation:pt}
For nonempty $J \subseteq \{1, \ldots, d+1\}$, let $l_J \subseteq \PP^d$ be the coordinate subspace
\begin{equation}
l_J = \{[x_1, \ldots, x_{d+1}] \in \PP^{d}: x_i = 0 \textrm{ if }i \in J^c\}. \nonumber
\end{equation}
We will label blow-ups along a coordinate subspace by the index of the center being blown-up. For example, $f_{J}: X_J \to \PP^d$ will denote the blow-up of $\PP^d$ along $l_J$. For proper transforms of linear subspaces under iterated blow-ups, we will in general abuse notation by not demarcating the proper transform, but rather indicating which proper transform is intended via the ambient variety. For example, we will write $l_{J} \subseteq X_{J'}$ for the proper transform of $l_J$ under the blow-up $f_{J'}$ (and all blow-ups preceding $f_{J'}$). An exception to this convention will be made when the focus is on how a subvariety behaves under proper transform (as, for example, in the proof of Proposition \ref{prop:M0nPBvsPT}). In such cases, we denote the proper transform of a subvariety $V \subseteq X$ under a blow-up $f_J: X_J \to X$ by $\widetilde{V}$.
\end{notation}
\noindent The benefit of the above labeling scheme is that the exceptional divisor in the Kapranov construction from the blow-up along $l_J$ corresponds to the usual labeling $E_J$ (see Definition \ref{def:KapBasisM}).

\begin{example}
\label{ex:Pd}
In our notation, the ray $\rho_{12}$ of the fan of $\PP^2$ in Figure \ref{fig:raysL3} (a) corresponds to the line $V(\rho_{12}) = \{[x_1, x_2, 0] \in \PP^2\}$, while the exceptional divisor obtained from blowing up $l_1 = [1,0,0]$ is $E_1 = V(\rho_1)$ of Figure \ref{fig:raysL3} (b).
\end{example}
\begin{figure}
	\begin{center}
                \subfigure[]{\fanptwo}
                \subfigure[]{\fanlthree}
                    \end{center}
\caption{Rays of fans of $\PP^2$ and $\LM_3$}
\label{fig:raysL3}	
\end{figure}
 Cox rings were first defined for toric varieties in \cite{MR1299003} (see also Chapter 5 of \cite{cls}). The Cox (or total coordinate) ring of the toric variety $X_\S$ is the polynomial ring
\begin{equation}
\cox(X_\S) = \CC[x_\rho: \rho \in \S(1)]. \nonumber
\end{equation}
This ring has a $\Pic(X_\S)$-grading defined by 
\begin{equation}
\deg( \prod_{\rho \in \S(1)} x_{\rho}^{a_\rho}) = [ \sum_{\rho \in \S(1)} a_\rho D_{\rho}]. \nonumber
\end{equation}

For $\alpha \in \Pic(X_\S)$, we label the $\alpha$-graded part of the Cox ring by $\cox(X_\S)_\alpha$. If a divisor $D = \sum a_\rho D_\rho$ has class $\alpha$, there exists a non-canonical isomorphism $H^0(X_\S, \cO_{X_\S}(D)) \to \cox(X_\S)_\alpha$ (see \cite{cls}, Sections 4.3 and 5.4, for more details and proofs).

The absence of a canonical identification between the $\alpha$-graded part of $\cox(X_\S)$ and global sections of a divisor whose class is $\alpha$ can be remedied by selecting divisors $D_1, \ldots, D_r$ whose classes form a basis for $\Pic(X_\S)$. With this choice, multiplication of sections is induced by multiplication of functions in $\CC(X_\S)$, bringing us to the more general definition of the Cox ring of a projective variety from \cite{MR1786494}.
\begin{defn}
\label{def:coxProj}
Let $X$ be a projective variety with a torsion-free Picard group satisfying $\Pic(X)_\QQ = N^1(X)_\QQ$. Let $D_1, \ldots, D_r$ be divisors whose classes form a basis of $\Pic(X)_\QQ$. The \emph{Cox ring} of $X$ with respect to this choice of divisors is
\begin{equation}
\cox(X) = \sum_{(m_1, \ldots, m_r) \in \ZZ^r} H^0(X, \cO_X(m_1 D_1 + \ldots + m_r D_r)) \nonumber,
\end{equation}
with multiplication given by multiplication of functions in $\CC(X)$, and the grading defined by the Picard group, as above.
\end{defn}
\noindent It is proved in \cite{MR1786494} that different choices of divisors yield non-canonically isomorphic Cox rings.

Before turning to the proof of Theorem \ref{thm:main}, we make two remarks about notation. For a divisor $D$ on a variety $X$ satisfying $\Pic(X)_\QQ = N^1(X)_\QQ$, we will use without further comment both the notation $[D]$ for the numerical equivalence class of $D$ in $N^1(X)_{\QQ}$ and $\cO_X(D)$ for the corresponding element of $\Pic(X)$.

Lastly, when referring to a boundary divisor $\D_J$ (or corresponding section variable $x_J$ in the Cox ring), we take for granted that $J \subseteq \{1, \ldots, n\}$ with $2 \leq |J| \leq n-2$.

\section{Permutohedral subrings in the Cox ring of $\M_{0,n}$}
\label{secCox:M0nLnm2}
To prove Theorem \ref{thm:main} and establish the polynomial permutohedral subrings of $\cox(\M_{0,n})$, we study the Kapranov's blow-up constructions of $\M_{0,n}$ and $\LM_{n-2}$, with special emphasis on the order of the blow-ups. To construct $\LM_{n-2}$ as a blow-up of $\PP^{n-3}$, we make use of the notational convention \ref{notation:pt}. We first blow up $l_1 = [1, 0, \ldots, 0]$, then the proper transform $l_2 \in X_1$, where $f_1 : X_1 \to \PP^{n-3}$ is the blow-up along $l_1$, continuing until we have blown-up $l_{n-2} \in X_{n-3}$. Next we blow up the proper transform of the line $l_{12} \subseteq X_{n-2}$, then the line $l_{13} \subseteq X_{12}$, continuing until all proper transforms of lines are blown up. We proceed in this way, blowing up proper transforms of coordinate subspaces in increasing order of dimension until all proper transforms of codimension two coordinate subspaces have been blown up. Note that this ordering respects the partial ordering by inclusion on the linear subspaces whose proper transforms are blown up. In other words, if we list the blow-up centers in the order they are blown up as $J_i$, then $l_{J_i} \subsetneq l_{J_j}$ only if $i < j$. This construction gives an explicit basis for $\Pic(\LM_{n-2})$. 
\begin{defn}
\label{def:KapBasisLM}
Let $t' : \LM_{n-2} \to \PP^{n-3}$ be the composition of blow-ups in the preceding paragraph. The \emph{Kapranov basis} of $\LM_{n-2}$ consists of the classes of the following divisors in $\LM_{n-2}$:
\begin{itemize}
\item the pull-back of a generic hyperplane in $\PP^{n-3}$, denoted $H'$, and
\item the (proper transforms of) exceptional divisors obtained by blowing up (the proper transforms of) $l_J$ for $J \subseteq \{1, \ldots, n-2\}$ and $1 \leq |J| \leq n-4$. We denote these divisors by $E'_{J}$.  
\end{itemize}
\end{defn}
\noindent The dash on these morphisms and classes is to distinguish these classes from their analogues for $\M_{0,n}$ to be discussed shortly. Kapranov's basis for $\LM_{n-2}$ implies, in particular, that the Picard number of $\LM_{n-2}$ is $2^{n-2} - n +1$.

The fan of $\LM_{n-2}$ is determined by the various star subdivisions of the fan for $\PP^{n-3}$ as described in Section \ref{secCox:background}. For $1 \leq |J| \leq n-4$, the ray $\rho_J$ determines (the proper transform of) the exceptional divisor arising from blowing up (the proper transform of) the coordinate subspace $l_J$. For $|J| = n-3$, the divisor associated to $\rho_J$ is the proper transform of the hyperplane $l_J$.

\begin{example}
\label{ex:blowupL3}The permutohedral space $\LM_3$ is the blow-up of $\PP^2$ in points $l_1 = [1,0,0], l_2 = [0,1,0]$, and $l_3 = [0,0,1]$. The rays of the fan $\S(\LM_3)$ are depicted in Figure \ref{fig:raysL3} (b). The fan structure is obvious, but it will be helpful for higher $n$ to note that a set of rays generates a cone of $\S(\LM_3)$ if and only if the indices of the rays are totally ordered under inclusion. For example, the rays $\{\rho_1, \rho_{12}\}$ generate a two-dimensional cone, but $\{\rho_1, \rho_2\}$ does not determine a cone.

By the discussion about proper transforms under toric blow-ups above,
\begin{align*}
[D_{23}] &=  [H'] - [E'_2] - [E'_3],\\
[D_{13}] &=  [H'] - [E'_1] - [E'_3], \\
[D_{12}] &=  [H'] - [E'_1] - [E'_2], \\
D_{1} = E'_1&, \,
D_{2}  = E'_2,  \textrm{ and }
D_{3}  = E'_3, 
\end{align*} 
where $E'_i$ denotes the divisor $V(\rho_i)$.
\end{example}


Kapranov constructed $\M_{0,n}$ from $\LM_{n-2}$ in \cite{MR1237834} by further blow-ups along non-torus-invariant linear subspaces. In addition to linear spans of the points $l_1, \ldots, l_{n-2}$, we take one more point in general position. For concreteness, we set $l_{n-1} = [1, \ldots, 1]$.

We first blow up the proper transform of the point $[1,\ldots,1]$, that is $l_{n-1} \in X_{3 \, \ldots \, n-2}$, and then blow up the proper transforms of all remaining linear centers containing $l_{n-1}$ in two stages: in the first round of blow-ups, which we call stage-$l_{n-2}^c$, we blow up proper transforms of linear centers containing $l_{n-1}$ but not $l_{n-2}$ in order of increasing dimension, as above, while in the second, labeled stage-$l_{n-2}$, we blow-up the remaining proper transforms of linear centers containing both $l_{n-2}$ and $l_{n-1}$, again, in order of increasing dimension. 

Note that this ordering of the blow-ups still respects the partial ordering by inclusion.
\begin{defn}
\label{def:KapBasisM}
Let $f: \M_{0,n} \to \LM_{n-2}$ be the composition of blow-ups involving $l_{n-1}$ above, and set $t = t' \circ f: \M_{0,n} \to \PP^{n-3}$. 

The \emph{Kapranov basis} of $\M_{0,n}$ consists of the classes of the following divisors in $\M_{0,n}$:
\begin{itemize}
\item the pullback under $t$ of a generic hyperplane in $\PP^{n-3}$, denoted by $H$;
\item the proper transforms of $E'_J$, where $J \subseteq \{1, \ldots, n-2\}$, $1 \leq |J| \leq n-4$; and
\item the (proper transforms of the) exceptional divisors obtained by blowing up the proper transforms of $l_J$, where $n-1 \in J$ and $1 \leq |J| \leq n-4$.
\end{itemize}
We denote divisors of the last two types by $E_J$.
\end{defn}
 
We next establish how the members of the Kapranov basis for $\M_{0,n}$ relate to the pull-backs of classes in the Kapranov basis for $\LM_{n-2}$ under the composition of blow-ups $f: \M_{0,n} \to \LM_{n-2}$. Before turning to the general case, we look at the simplest non-trivial example of $\M_{0,5}$ and $\LM_{3}$.
\begin{example}
\label{ex:M05pbpt} 
To obtain $\M_{0,5}$ from $\LM_{3}$, we further blow-up $p_4 = [1,1,1]$. Since $E_4$ is disjoint from $f^*(E'_i)$, $i=1, \ldots, 3$, it follows that the proper transforms of the $E'_i$ equal their pull-backs under $f: \M_{0,5} \to \LM_3$ for $i=1, 2,3$, that is,
\begin{equation}
E_i = f^*(E'_i), \quad i=1, \ldots, 3 \nonumber.
\end{equation}
\end{example}

For $n \geq 6$, the exceptional divisors from the stages $l_{n-2}^c$ and $l_{n-2}$ described before Definition \ref{def:KapBasisM} are not disjoint from pull-backs of exceptional divisors of the preceding stages. Nevertheless, we show that the same relationship between pull-backs and proper transforms of exceptional divisors holds.
\begin{prop}
\label{prop:M0nPBvsPT}
For every $J \subseteq \{1, \ldots, n-2\}$, $1 \leq |J| \leq n-4$,
\begin{equation}
f^*(E_J')  = \widetilde{E_J'} = E_J. \nonumber
\end{equation}
\end{prop}
\noindent Note that the second equality is definitional.
For the proof we use a general characterization of proper transforms from \cite{MR1644323}, B.6:
\begin{prop}
\label{prop:ehBlowUp}
Let $Z$ be a smooth subvariety of a variety $Y$, and let $f_Z: Bl_Z(Y) \to Y$ be the blow-up of $Y$ along $Z$. If $V$ is a smooth subvariety of $Y$ containing $Z$, then the proper transform $\widetilde{V}$ is the blow-up of $V$ along $Z$, that is, $\widetilde{V} = Bl_{V \cap Z}V \to V$.
\end{prop}
A second ingredient is the notion of \emph{clean} intersections, as formulated in \cite{MR2595553}. Let $X$ be a nonsingular variety, and let $A$ and $B$ be nonsingular subvarieties. Denote by $T_A$ the total space of the tangent bundle of $A$, here considered as a subbundle of $T_X$. For $a \in A$ we denote by $T_{A,a}$ the tangent space of $A$ at the point $a$, taken as a subspace of $T_{X,a}$. 
\begin{defn}
\label{define:clean}
The subvarieties $A$ and $B$ are said to intersect \emph{cleanly} if
\begin{enumerate}
\item the set-theoretic intersection $A \cap B$ is a nonsingular subvariety of $X$, and
\item $T_{A\cap B,y} = T_{A, y} \cap T_{B, y}$ for all $y \in A \cap B$.
\end{enumerate}
\end{defn}
For example, two lines $l$ and $l'$ in $\PP^3$ will always intersect cleanly, even though their intersection is never transverse: if $l$ and $l'$ are skew, then they satisfy the definition of clean intersection trivially, while if $l$ and $l'$ meet at a point $x$, the intersection $T_{l, x} \cap T_{l', x}$ is the trivial vector space, which is the tangent space to the subvariety $x$, and finally, if $l = l'$, the criteria for clean intersection are clearly satisfied. More generally, if $l_J$ and $l_{J'}$ are linear subspaces of $\PP^m$, their intersection is also clean. 

Clean intersections behave nicely under blow-ups. The following is from Lemma 2.9 in \cite{MR2595553} (part \ref{item:Li2.9ii} is, however, a standard result: see \cite{MR1644323}, Appendix B.6). Note that we will now denote the proper transforms by a $\sim$.
\begin{lemma}
\label{lemma:cleanPTunderBU}
For $A$, $B$, $C$, and $F$ nonsingular subvarieties of $X$, a nonsingular variety, let $f_F: Bl_F(X) \to X$ be the blow-up of $X$ along $F$ with exceptional divisor $E$.
\begin{enumerate}
\item 
\label{item:Li2.9ii}
If $A$ and $B$ intersect cleanly, with $A \nsubseteq B$, $B \nsubseteq A$ such that $F = A \cap B$, then $\widetilde{A} \cap \widetilde{B} = \emptyset$.
\item 
\label{item:Li2.9i}
If $A \supsetneq F$, then $\widetilde{A}$ and $E$ intersect transversally.
\item
\label{item:Li2.9v}
If $F \subseteq A$ with both $A$ and $F$ intersecting $B$ transversally, then $\widetilde{A}$ and $\widetilde{B}$ intersect transversally.
\end{enumerate}
\end{lemma}
\begin{proof}[Proof of Propostion \ref{prop:M0nPBvsPT}]
Since the pull-back and proper transform of a subvariety coincide if the blow-up is along a center disjoint from the subvariety, we may restrict attention to subspaces $l_J$ and $l_{J'}$ such that $J \cap J' \neq \emptyset$. By Lemma \ref{lemma:cleanPTunderBU} \ref{item:Li2.9ii}, if $J \nsubseteq J'$ and $J' \nsubseteq J$, then once the proper transform of $l_{J \cap J'}$ is blown-up, the proper transforms of $l_J$ and $l_{J'}$ will be disjoint, as will all successive proper transforms and inverse images of $l_J$ and $l_{J'}$. Due to the partial ordering of blow-ups, we therefore need only consider how the pull-back and proper transform of the exceptional divisor of the blow-up of (the proper transform of) $l_J$ relate under the blow-up of the proper transform of $l_{J'}$ for $J \subsetneq J'$.

We first consider $J_1 \subseteq J_2 \subseteq J_3$, with $|J_3| = |J_2| + 1 = |J_1| + 2$. Let $f_{J_2}: X_{J_2} \to X_{J_1}$ be the blow-up along $l_{J_2} \subseteq X_{J_1}$ with exceptional divisor $E$. For the blow-up $f_{J_3}: X_{J_3} \to X_{J_2}$ along $l_{J_3} \subseteq X_{J_2}$, we want to show that $f_{J_3}^{*}(E) = \widetilde{E}$. The generic point of $E$ is disjoint from the center of the blow-up $l_{J_3}$, so since $f_{J_3}$ is an isomorphism over such points, it suffices to show that $f_{J_3}^{-1}(E) = \widetilde{E}$. 

By Proposition \ref{prop:ehBlowUp}, $\widetilde{E}$ is the blow-up of $E$ along $l_{J_3} \cap E$. Denote this blow-up by $\phi: \widetilde{E} \to E$, and its exceptional divisor by $F$. Since $\phi^{-1}(E \setminus l_{J_3}) = f_{J_3}^{-1}(E \setminus l_{J_3}) = f_{J_3}^{-1}(E) \setminus f_{J_3}^{-1}(l_{J_3})$, we are finished in this case if $F = f_{J_3}^{-1}(l_{J_3} \cap E)$. But $F \subseteq f_{J_3}^{-1}(l_{J_3} \cap E)$, and $F$ is a projective bundle over $l_{J_3} \cap E$, with each fiber a projective space of dimension equal to the codimension of $l_{J_3} \cap E$ in $E$. By Lemma \ref{lemma:cleanPTunderBU} \ref{item:Li2.9i}, $E$ meets $l_{J_3}$ transversely, so for $q \in l_{J_3} \cap E$, the fiber $F_q$ has dimension $n - |J_3| - 4$. On the other hand, $f_{J_{3}}^{-1}(q)$ is a fiber of the projectivized normal bundle $\PP(\cN_{X_{J_2} \setminus l_{J_3}})$. The dimension of $f_{J_3}^{-1}(q)$ is also $n - |J_3| - 4$, so we have an inclusion of projective spaces of the same dimension, giving $F_q = f_{J_3}^{-1}(q)$. It follows that $F = f_{J_3}^{-1}(l_{J_3} \cap E)$, as desired.

We now consider a chain of inclusions $J_1 \subseteq J_2 \subseteq J_3 \subseteq \ldots \subseteq J_k$, where again, $|J_j| = |J_{j-1}| + 1$. Abusing notation as usual, let $E \subseteq X_{J_{k-1}}$ be the proper transform of the exceptional divisor of $f_{J_2}: X_{J_2} \to X_{J_1}$. If $\widetilde{E}$ is now the proper transform under $f_{X_{J_k}}: X_{J_{k}} \to X_{J_{k-1}}$, the blow up along $l_{J_k} \subseteq X_{J_{k-1}}$, the proposition is proved once we show $\widetilde{E} = f_{J_k}^{-1}(E)$. Since every center blown up before $l_{J_k}$ is contained in $l_{J_k}$, Lemma \ref{lemma:cleanPTunderBU} \ref{item:Li2.9v} implies that $\widetilde{E}$ and $l_{J_k}$ intersect transversally. The rest of the proof now proceeds identically to the initial case.
\end{proof}
Since the hyperplane class in $\M_{0,n}$ is by definition the pull-back of the hyperplane class on $\LM_{n-2}$, Proposition \ref{prop:M0nPBvsPT} implies:
\begin{cor}
\label{cor:pbFormula}
The pull-back of $[D] = h [H'] + \sum_{\substack{1 \leq |J| \leq n-4, \\ J \subseteq \{1, \ldots, n-2\}}} e_J [E'_J] \in \Pic(\LM_{n-2})$ to $\Pic(\M_{0,n})$ is given by
\begin{equation}
f^*[D] = h [H] + \sum_{\substack{1 \leq |J| \leq n-4 \\ J \subseteq \{1, \ldots, n-2\}}} e_J [E_J]. \nonumber
\end{equation}
\end{cor}

Proposition \ref{prop:M0nPBvsPT} almost gives an identification between the divisors $D_J$ of $\LM_{n-2}$ and the boundary divisors $\D_J$ of $\M_{0,n}$. The minor obstacle to this identification is that the order of blow-ups used in the majority of the literature is not the one we used in defining the Kapranov basis (Definition \ref{def:KapBasisM}); instead the ordering due to Hassett is generally used (\cite{MR1957831}). Hence it is not immediately obvious (but also not difficult to prove) that the hyperplane classes and exceptional divisors arising from the different blow-up orderings are interchangeable.

For the remainder of this section only, we distinguish the varieties resulting from the two orderings of the blow-ups by $\M^k_{0,n}$ and $\M^h_{0,n}$ for the Kapranov and Hassett constructions, respectively. Likewise, we denote the resulting bases of the Picard groups by $\cB^k = \{[H^k], [E^k_J]\}$ and $\cB^h = \{[H^h], [E^h_J]\}$. We now prove that the two ordering of the blow-ups are related by an isomorphism that takes the basis $\cB^h$ to the basis $\cB^k$.

We begin with a basic observation.
\begin{lemma}
\label{lemma:disjointBU}
Let $X$ be a smooth variety with disjoint, closed subvarieties $A_1$ and $A_2$. Let
$f_1 : X_1 \to X$ be the blow-up of $X$ along $A_1$, and let $f_2: X_2 \to X$ be the blow-up of $X$ along $A_2$. We denote the proper transform of $A_2$ under $f_1$ by $\widetilde{A_2}^{f_1}$, and likewise, the proper transform of $A_1$ under $f_2$ by $\widetilde{A_1}^{f_2}$. Let $g_2: X_{21} \to X_1$ be the blow-up of $X_1$ along $\widetilde{A_2}^{f_1}$, and let $g_1: X_{12} \to X_2$ be the blow-up of $X_2$ along $\widetilde{A_1}^{f_2}$. 

If $E_2$ is the exceptional divisor of $g_2$, and $E_1$ the proper transform under $g_2$ of the exceptional divisor of $f_1$, and likewise $F_1$ is the exceptional divisor of $g_1$ and $F_2$ the proper transform under $g_1$ of the exceptional divisor of $f_2$, then there exists an isomorphism
\begin{equation}
\phi: X_{21} \stackrel{\cong}{\to} X_{12}\nonumber.
\end{equation} 
such that
\begin{equation}
\phi^*(F_i) = E_i \textrm{ for }i=1,2. \nonumber
\end{equation}
\end{lemma}
The first part of the lemma is a standard result, and is proved in greater generality in \cite{MR2595553}. It is the second part that enables us to prove that the Kapranov and Hassett orderings result in the same basis for $\Pic(\M_{0,n})$ up to isomorphism.
\begin{proof}
First, since $A_1$ and $A_2$ are disjoint, $\widetilde{A_j}^{f_i} = f_i^{-1}(A_j)$ for $i,j = 1,2$, $i \neq j$. The proof involves finding open covers of $X_{12}$ and $X_{21}$, plus isomorphisms of the elements of the covers that agree on overlap. 
\begin{displaymath}
\xymatrix{
X_{21}= Bl_{\widetilde{A_2}^{f_1}} (X_1)\ar@{.>}[rr]^{\phi_i} \ar[d]^{g_2}&& 
X_{12} = Bl_{\widetilde{A_1}^{f_2}}(X_2) \ar[d]^{g_1}\\
X_1 = Bl_{A_1}(X) \ar[dr]_-{f_1} &   & X_2 = Bl_{A_2}(X) \ar[dl]^-{f_2}\\
& X &}
\end{displaymath}
Define $U_1 = X_{21} \setminus (g_2^{-1} \circ f_1^{-1}(A_1))$ and $U_2 = X_{21} \setminus (g_2 ^{-1} \circ f_1^{-1}(A_1))$. Likewise, set $V_1 = X_{12} \setminus (g_1^{-1} \circ f_2^{-1} (A_1))$ and $V_2 = X_{12} \setminus (g_1^{-1} \circ f_2^{-1}(A_2))$. Then $\{U_1, U_2\}$ and $\{V_1, V_2\}$ define open covers $X_{21}$ and $X_{12}$, respectively.

Note that $g_2|_{U_1}: U_1 \to X_1 \setminus f^{-1}(A_1)$ is the blow-up of $X_1 \setminus f_{1}^{-1}(A_1)$ along $f_1^{-1}(A_2)$, and $f_1$ defines an isomorphism between $X_1 \setminus f_{1}^{-1}(A_1)$ and $X \setminus A_1$. Since $g_2^{-1} \circ f_1^{-1}(A_2) = E_2$ is a Cartier divisor, by the proof of the universal property of blowing-up (see \cite{MR0463157}), the unique morphism between $U_1$ and $X_2 \setminus f_2^{-1}(A_1)$ factoring $f_2$ is an isomorphism. Composing with the inverse of the isomorphism $g_1|_{V_1}:V_1 \to X_2 \setminus (f_2^{-1}(A_1))$ determines an isomorphism $\phi_1:U_1 \to V_1$. We similarly obtain an isomorphism $\phi_2: U_2 \to V_2$.

By definition, $(f_1 \circ g_2)|_{U_1 \cap U_2}: U_1 \cap U_2 \to X \setminus (A_1 \cup A_2)$ is an isomorphism, as is $(f_2 \circ g_1)|_{V_1 \cap V_2}: V_1 \cap V_2 \to X \setminus (A_1 \cup A_2)$. Moreover, $(f_2 \circ g_1)^{-1} \circ (f_1 \circ g_2)|_{U_1 \cap U_2}$ agrees with $\phi_1$ and $\phi_2$ on $U_1 \cap U_2$. Hence $\phi_1$ and $\phi_2$ glue together to give the desired isomorphism $\phi$. By construction, $\phi^{-1}(F_i) = E_i$ for $i=1,2$.
\end{proof}

Hassett's ordering of blow-ups follows the dimension of the linear centers in $\PP^{n-3}$. Specifically, we first blow-up the points (or proper transforms of) $l_1, \ldots, l_{n-1}$, then the proper transforms of the lines $l_{12}, \ldots, l_{n-2 \, n-1}$, continuing until all proper transforms of $l_J$ with $|J| = n-4$ have been blown up. 
\begin{prop}
\label{prop:KapHassBases}
There exists an isomorphism $\phi: \M^k_{0,n} \to \M^h_{0,n}$ such that 
\begin{equation}
\phi^*(H^h) = H^k \textrm{ and }\phi^*(E^h_J) = E^k_J. \nonumber 
\end{equation}
\end{prop}
The existence of such an isomorphism is well-known (see \cite{MR1957831}, \cite{MR2595553}, and \cite{mustafin}), and the claim about pull-backs of basis elements could be proven by a modular interpretation of the Kapranov basis, but a direct proof seems preferable.
\begin{proof}
In the Kapranov ordering, $l_{n-1}$ is disjoint from all linear subspaces $l_J$ whose proper transforms are blown-up before it. Hence all proper transforms of $l_{n-1}$ and such $l_J$ are also disjoint, so by Lemma \ref{lemma:disjointBU}, we may interchange the blowing-up of the proper transform of $l_{n-1}$ successively with each of the blow-ups preceding it. In particular, we may blow up $l_{n-1}$ after blowing up $l_{n-2}$ to match the Hassett ordering.

In general, suppose inductively that we have brought the proper transform of the $i^{th}$ linear subspace $l_{J_i}$ into agreement with the Hassett ordering for $i < j$. For $l_{J_i}$, $i<j$, such that $|J_i| \geq |J_j|$, we have $l_{J_i} \nsubseteq l_{J_j}$ and $l_{J_j} \nsubseteq l_{J_i}$. Lemma \ref{lemma:cleanPTunderBU} implies that, after blowing up the proper transform of $l_{J_i \cap J_j}$, the proper transforms of $l_{J_i}$ and $l_{J_j}$ are disjoint. Since $|J_i \cap J_j| < |J_j|$, we may switch the order of blow-ups of the proper transform of $l_{J_j}$ successively with each $l_{J_i}$ such that $|J_i| \geq |J_{j}|$. In particular, we may change the order so that the proper transform of $l_{J_j}$ is as in the Hassett ordering. Applying Lemma \ref{lemma:disjointBU} proves the claim about the respective exceptional divisors. The claim about pull-backs of a generic hyperplane in $\PP^{n-3}$ follows from the result about exceptional divisors plus the gluing of Lemma \ref{lemma:disjointBU}. 
\end{proof}

Via Propositions \ref{prop:M0nPBvsPT} and \ref{prop:KapHassBases}, we can identify boundary divisors of $\LM_{n-2}$ with boundary divisors $\D_J$ in $\M_{0,n}$.

\begin{defn}
\label{defnCox:lmD}
For $J \subseteq \{1, \ldots, n-2\}$, $1 \leq |J| \leq n-3$, let $\D_{J \cup \{n \}}'$ be the torus-invariant divisor $V( \langle \rho_J \rangle_{\geq 0})$.
\end{defn}
\noindent If $1 \leq |J| \leq n-4$, then $\D_{J \cup \{n\}}' = E'_J$, while for $|J| = n-3$, $\D_{J \cup \{n \}}'$ is the proper transform of the line containing all $l_{J'} \subseteq \PP^{n-3}$, with $J' \subsetneq J$. As with boundary divisors in $\M_{0,n}$, we will identify $\D_{J}'$ and $\D_{J^c}'$.
\begin{cor}
\label{corCox:pbBoundary}
For every boundary divisor class $[\D_{J}']\in \Pic( \LM_{n-2})_\QQ$,
\begin{equation}
f^*([\D_{J}']) = [\D_J] \in \Pic(\M_{0,n})_\QQ \nonumber.
\end{equation}
\end{cor}

The final ingredient for the proof of Theorem \ref{thm:main} is an isomorphism between the global sections of a divisor in $\LM_{n-2}$ and the global sections of its pull-back in $\M_{0,n}$. That the induced map on global sections is an isomorphism is a standard result, and holds also for much more general situations.

\begin{lemma} 
\label{lemmaCox:pbH0}
For $f: \M_{0,n} \to \LM_{n-2}$ as above, and $D$ a divisor on $\LM_{n-2}$,
\begin{equation}
f^*: H^0(\LM_{n-2}, D) \stackrel{\cong}{\to} H^0(\M_{0,n}, f^*(D)) \nonumber.
\end{equation}
\end{lemma}

\begin{proof}
We first set $\mathcal{L} = \cO_{\LM_{n-2}}(D)$ (note that $\LM_{n-2}$ is smooth). Since $f$ is birational and projective, and both $\M_{0,n}$ and $\LM_{n-2}$ are integral, noetherian schemes with $\LM_{n-2}$ normal, it follows, as in the proof of Zariski's main theorem (see \cite{MR0463157}, Corollary III.11.4), that $f_*(\cO_{\M_{0,n}}) = \cO_{\LM_{n-2}}$. By the projection formula, $f_* f^* \mathcal{L} \cong (f_* \cO_{\M_{0,n}}) \otimes \mathcal{L}$, hence $ H^0(\M_{0,n}, f^*(\mathcal{L})) \cong H^0(\LM_{n-2}, f_* f^*(\mathcal{L})) \cong H^0(\LM_{n-2}, \mathcal{L})$.
\end{proof}

\begin{proof}[Proof of Theorem \ref{thm:main}]
Lemma \ref{lemmaCox:pbH0} shows that the induced map $f^*: \cox(\LM_{n-2}) \to \cox(\M_{0,n})$ is injective, with image 
\begin{equation*}
\bigoplus_{(m_h, m_J) \in \ZZ^{2^{n-2} - n +1}} H^0(\M_{0,n}, m_h H + \sum_{J \subsetneq \{1, \ldots, n-2\}} m_J E_J),
\end{equation*}
Corollary \ref{cor:pbFormula} shows that $f^*$ respects the $\Pic$-grading, and Corollary \ref{corCox:pbBoundary} shows that, on the level of boundary section generators, the map is given by $x_{\Delta'_J} \to x_{\Delta_J}$.
\end{proof}

Thus far, we have been discussing a single permutohedral space $\LM_{n-2}$, where, in relation to $\M_{0,n}$, we take the points labeled by $(n-1)$ and $n$ to be the poles (or equivalently, the $(n-1)^{\textrm{st}}$ marked point corresponds to the non-toric point $l_{n-1} \in \PP^{n-3}$, and the $n^{\textrm{th}}$ marked point is the `moving' point). The choice of poles, however, is arbitrary, since permuting the marked points results in isomorphic copies of $\M_{0,n}$ . Hence there are $\binom{n}{2}$ permutohedral spaces that have the same relationship with $\M_{0,n}$ as that described in this section by choosing different pairs of points for the poles; we denote these varieties as $\LM_{n-2}(i,j)$.

Since $\cox(\LM_{n-2}(i,j))$ is a polynomial ring in the variables $x_{\D'_J}$, where $i \in J$, $J \subseteq (\{1, \ldots, n\} \setminus \{j\})$, and $2 \leq |J| \leq n-3$, Theorem \ref{thm:main} establishes $\binom{n}{2}$ polynomial rings $\cox(\LM_{n-2})(i,j)$ sitting inside $\cox(\M_{0,n})$, i.e., whose preimage in the polynomial ring of generators (assuming finite generation) intersects the ideal of relations trivially.

Before considering an application to Pl\"ucker relations in $\cox(\M_{0,n})$, we note that Theorem \ref{thm:main} gives a partial answer to the Riemann-Roch problem for $\M_{0,n}$: for any divisor with a section in the image of the map $f^*$ (or its analogue by choosing different poles) the dimension of the space of global sections equals that of the corresponding torus-invariant divisor in $\LM_{n-2}(i,j)$, which can be calculated by counting lattice points (see for example Section 4.3 of \cite{cls}). It would be interesting to obtain a closed formula for these dimensions.
\selectlanguage{english}


\section{Pl\"ucker relations and pull-backs of hyperplane classes}
\label{secCox:CoxFull}
We conclude by applying Theorem \ref{thm:main} to show that relations in certain degrees of $\cox(\M_{0,6})$ defined by pull-backs from projective spaces are generated by the Pl\"ucker relations (defined below). The result is given for $n=6$, but the proof is such that it would hold for higher $n$ given a technical assumption on additional (yet to be established) generators. The proof mirrors part of the proof of the Batyrev-Popov conjecture about quadratic generation of the Cox ring of del Pezzo surfaces given in \cite{MR2529093}. It is our hope that such an inductive approach could establish the correctness of a conjectural presentation of $\cox(\M_{0,6})$. 

We begin by defining degrees $[F_{J,m}] \in \Pic(\M_{0,n})$ as pull-backs of hyperplane classes of projective spaces under compositions of forgetful and Kapranov morphisms, then calculate the dimension of the space of global sections of these degrees. This dimension turns out to be smaller than the number ways that $[F_{J,m}]$ can be represented as an effective sum of boundary divisors, hence the number of generators of this part of $\cox(\M_{0,n})$ is bigger than its dimension, and there must be relations in this $\Pic(\M_{0,n})$-degree. Finally, we show that all relations in boundary section variables in the degrees $[F_{J,m}]$ are generated by the Pl\"ucker relations when $n=6$. For $\M_{0,5}$ the Pl\"ucker relations generate the entire ideal of relations (see \cite{MR2029863}), while for $n=6$, there are known non-Pl\"ucker-generated relations (see Example 6.6 of \cite{hilbChow}). 

We now fix some additional notation and make a few further remarks about the Kapranov construction. An important class of morphisms between moduli spaces of stable pointed rational curves is the \emph{forgetful} morphisms, corresponding to forgetting a subset of the marked points and then, if necessary, stabilizing (see \cite{MR702953} or \cite{MR1034665} for further details). Since we will be forgetting varied subsets $J \subseteq \onetn$, we will keep track of which points are remembered by labeling the target space as $\M_{0,\onetn \setminus J}$.

Kapranov showed in \cite{MR1203685} that the morphism $t: \M_{0,n} \to \PP^{n-3}$ is induced by the psi-class $\psi_n$ of the `moving' point. In Definition \ref{def:KapBasisM} we chose the $n^{\textrm{th}}$ point to be moving, but the same holds for any $m \in \{1, \ldots, n\}$. We denote the corresponding morphism by $t_m: \M_{0,n} \to \PP^{n-3}$. In particular, the hyperplane class of the Kapranov basis for $m$ as the moving point is $\psi_m$. To define the divisor classes $[F_{J,m}]$, fix a subset $J \subseteq \onetn$ with $0 \leq |J| \leq n-4$, and choose $m \notin J$. 
For simplicity, we will choose $J \subseteq \{1, \ldots, n-1\}$, and take $m=n$, since the general situation can be obtained from this one the action of the symmetric group $S_n$ on $\M_{0,n}$: for $i \neq j$, the change of basis on $\Pic(\M_{0,n})$ is given by 
\begin{align}
\psi_j &= (n-3)\psi_i - \sum_{\substack{J \subseteq (\onetn \setminus\{i,j\}) \\ 1 \leq |J| \leq n-4}}(n-|J| - 3) [\D_{J \cup \{i\}}], \\
[\D_{J \cup \{j\}}] &= \left\{ \begin{array}{ll}
				[\D_{J \cup \{j\}}] & \textrm{if }i \in J,\\
				\textrm{$[$} \D_{(J \cup \{j\})^c} \textrm{$]$} & \textrm{if } i \notin J, |J| > 1,\\
\psi_i - \sum\limits_{T \subsetneq (J \cup \{i,j\})^c} [\D_{T \cup \{i\}}] & \textrm{if }i \notin J \textrm{ and }|J| = 1
\end{array} \right. \nonumber
\end{align}
\noindent (see \cite{brunoMella} for the first formula, though it also follows easily from the expressions for psi-classes determined in \cite{MR1733327}).

Define $\phi_{J,n} = t_n \circ \pi_{J}: \M_{0,n} \to \PP^{n - |J| - 3}$; this morphism is induced by the divisor class $[F_{J,n}] = \phi_{J,n}^*(\cO_{\PP^{n - |J| - 3}}(1))$. We can thus determine divisor representatives by studying how the hyperplane class on $\PP^{n-|J| - 3}$ pulls-back under $\phi_{J,n}$, or, equivalently, how psi-classes pull-back under the forgetful morphisms $\pi_J$. We begin with a lemma from \cite{MR1733327}.

\begin{lemma}
\label{lemma:acpsi}
For $q \in P \subseteq \onetn$ and $p \in P- \{q\}$ the pull-back of $\psi_p$ under $\pi_{q}: \M_{0,P} \to \M_{0, P\setminus\{q\}}$ is given by  
\begin{gather}
\label{eq:pbpsi}
\pi_{q}^*(\psi_p) = \psi_p - [\Delta_{pq}]
\end{gather}  
\end{lemma}
\noindent A second fact of use is a special case of a lemma of Keel from \cite{MR1034665}.
\begin{lemma}
For $q \in P \subseteq \onetn$ and $T \subseteq P\setminus\{q\}$, the pullback of $[\D_T]$ under $\pi_q: \M_{0,P} \to \M_{0, P\setminus\{q\}}$ is
\begin{gather}
\label{eq:pbdelta}
\pi^*_{q}([\D_T]) = [\D_T] + [\D_{T \cup \{q\}}].
\end{gather}
\end{lemma}
\noindent The final ingredient is that if $J = \{j_1, \ldots, j_k\}$ (allowing the possibility $J = \emptyset$), then the forgetful morphism $\pi_J$ decomposes as $\pi_J = \pi_{j_1} \circ \ldots \circ \pi_{j_k}$ (or as any re-ordering). Combining, we obtain the following.
\begin{lemma}
 For $J \subseteq \{ 1, \ldots, n-1\}$, the pullback of $\psi_n$ under 
$\pi_{J}: \M_{0,\onetn} \to \M_{0,\onetn \setminus J}$
is
\begin{gather}\label{eq:pbJ}
\pi_J^*(\psi_n) = \psi_n - \sum_{T \subseteq J} [\D_{T \cup \{n\}}]
\end{gather}
\end{lemma}
\begin{proof}
The case $|J|=1$ is Lemma \ref{lemma:acpsi} of Arbarello and Cornalba, so we assume inductively the validity of Equation (\ref{eq:pbJ}) for all $J \subseteq \{1, \ldots, n-1\}$, $|J| = k < n-1$. Let $q \in \{1, \ldots, n-1\} \setminus J$. Then
\begin{align}
\pi^*_{J \cup \{q\}}(\psi_n) &= \pi^*_{q} \circ \pi^*_J (\psi_n) \nonumber \\
	&=\pi^*_q (\psi_n - \sum_{T \subseteq J} [\D_{T \cup \{n\}}]) \nonumber \\
	&=\psi_n - [\D_{qn}] - \sum_{T \subseteq J} ([\D_{T \cup \{n\}}] + [\D_{T \cup \{q,n\}}]) \nonumber \\
	&= \psi_n - \sum_{T \subseteq (J \cup \{q\})} [\D_{T \cup \{n\}}] \nonumber,
\end{align}
as desired.
\end{proof}
Therefore, with respect to the Kapranov basis,
\begin{gather}
\label{eq:FIKap}
[F_{J,n}] = [H] - \sum_{T \subseteq J} [E_T].
\end{gather}

Before considering how to represent the divisor classes $[F_{J,n}]$ as effective sums of boundary, we show that these classes give a basis for $\Pic(\M_{0,n})_\QQ$. 
\begin{prop}
The collection of divisor classes,
\begin{gather}
\{[F_{J,n}]: J \subseteq \{1, \ldots, n-1\}, 0 \leq |J| \leq n-4 \}, \nonumber
\end{gather}
determines a basis of $\Pic(\M_{0,n})_\QQ$.
\end{prop}
\begin{proof}
The first thing to note is that the number of divisor classes $[F_{J,n}]$ is
\begin{align*}
1 + \binom{n-1}{2} + \binom{n-1}{3} + \ldots + \binom{n-1}{n-4}
& = 2^{n-1} - \binom{n}{2} - 1,
\end{align*}
which is precisely the dimension of $\Pic(\M_{0,n})_\QQ$. It therefore suffices to show that these classes are linearly independent. 

Set $\rho = \dim \Pic(\M_{0,n})_\QQ$ and fix an isomorphism $\Pic(\M_{0,n})_\QQ \cong \QQ^\rho$ by choosing the Kapranov basis, ordered in the usual way: 
\begin{gather*}
\cB_K = \{[H], [E_1], \ldots, [E_{n-1}], [E_{12}], [E_{13}], \ldots [E_{4 \ldots n-1}]\}.
\end{gather*}
We similarly order the set of classes $[F_{J,k}]$ as 
\begin{align*}
\cB_F = \{&[H], [H] - [E_1], [H] - [E_2], \ldots, \\
&[H] - [E_4] - [E_5] - \ldots - [E_{n-1}] - \ldots - [E_{4 \ldots n-1}] \}.
\end{align*}
Writing each element of $\cB_F$ in the coordinates given by the Kapranov basis $\cB_K$, we see that the $i^{\mathrm{th}}$ coordinate of the $i^{\mathrm{th}}$ element of $\cB_F$ will be -1 (except for $i=1$, when the coordinate is just 1), with all $j^{\mathrm{th}}$ coordinates, $j>i$, equal to 0. Hence the matrix whose columns are the coordinate vectors of the $[F_{J,n}]$ is upper triangular, with the $(i,i)$ entry equal to 1 if $i=1$, and $-1$ otherwise, thus showing that the $[F_{J,n}]$ are linearly independent.
\end{proof}

To obtain integral effective sums of boundary divisors whose class is $[F_{J,n}]$, we vary the representative of $[H]$, as in the second set of expressions in the following dictionary among boundary divisors, exceptional divisors, and the hyperplane class of $\M_{0,n}$ (see \cite{MR1882122}): 
\begin{align}
\D_{J \cup \{n\}} &= E_J, \textrm { if } 1 \leq |J| \leq n-4, \nonumber\\
[\D_{J \cup \{n\}}] & = [H] - \bigg(\sum_{J' \subsetneq J}[E_{J'}] \bigg), \textrm{ if } |J| = n-3. \nonumber
\end{align}

If $\{a,b\} = (J \cup \{n\})^c$ for $|J| = n-3$, the second set of equalities can be rewritten as
\begin{gather}
\label{eq:dictionaryDab}
[\D_{ab}] = [H] - \bigg( \sum_{J' \subsetneq \{a,b,n\}^c} [E_{J'}]   \bigg).
\end{gather}

\begin{prop}
\label{prop:effFJn}
All effective integral representations of the classes $[F_{J,n}]$ as boundary divisors are obtained by varying the representative of the hyperplane class as in Equation (\ref{eq:dictionaryDab}). In particular, there are $\binom{n-|J| - 1}{2}$ representatives of $[F_{J,n}]$ as an integral, effective sum of boundary divisors.
\end{prop}
\begin{proof}
Let $\qbnd$ denote the vector space with basis elements indexed by the boundary divisors of $\M_{0,n}$, and $\orthb$ the first orthant, which corresponds to all effective sums of boundary divisors with rational coefficients (we use throughout the identification $\D_T = \D_{T^c}$). By \cite{MR1034665}, the map $\cl: \qbnd \to \Pic(\M_{0,n})_\QQ$, defined by taking a sum of boundary divisors to its class, is surjective with the kernel having dimension $\binom{n}{2} - n$. 

We label the representative for the hyperplane class from Equation (\ref{eq:dictionaryDab}) as $H_{ab}$, and the corresponding element of $\orthb$ as $h_{ab}$, which we write as
\begin{gather}
h_{ab} = \sum_{\substack{a,b \in T^c , \\ n \in T}} \D_T.
\end{gather}
\noindent By the action of the symmetric group, we may assume that $\{a,b\} = \{1,2\}$. 

Consider the following elements of $\ker (\cl)$: for $a \in \{1,2\}$ and $x \in \{1,2,n\}^c$, define $r_{ax} = h_{ax} - h_{12}$, while for $\{y,z \} \subseteq \{1,2,n\}^c$, define $r_{yz} = h_{yz} - h_{12}$. Note that the cardinality of this set is $2(n-3) + \binom{n-3}{2} = \binom{n}{2} - n$. Hence to show that this collection determines a basis of $\ker(\cl)$, it suffices to show linear independence.

Suppose that
\begin{gather}
\sum_{\substack{a \in \{1,2\},\\ x \in \{1,2,n\}^c}} c_{ax} r_{ax} + \sum_{y<z \in \{1,2,n\}^c} c_{yz} r_{yz} = 0 \nonumber
\end{gather}
in $\qbnd$. The coefficient of $\D_{ax}$ in the above expression is $c_{ax}$, while the coefficient of $\D_{yz}$ is $c_{yz}$, hence all coefficients vanish.

Assume that $J \cap \{1,2\} = \emptyset$ (recall that $|J| \leq n-4$), so that $f_{J,n} = h_{12} - \sum_{T \subseteq J} \D_T$ is effective. Now consider $d \in \orthNb$ such that $d - f_{J,n} \in \ker (\cl)$, thus giving
\begin{gather}
\label{eq:effFJn}
d = h_{12} - \sum_{T \subseteq J} \D_{J \cup \{n\}} + \sum_{\substack{a \in \{1,2\},\\ x \in \{1,2,n\}^c}} c_{ax} r_{ax} + \sum_{y<z \in \{1,2,n\}^c} c_{yz} r_{yz}
\end{gather}
for unique $c_{ax}$ and $c_{yz}$. The proof of linear independence above implies that all coefficients are integral and non-negative, since none of the boundary divisors $\D_{ax}, \D_{yz}$ above appears in $f_{J,n}$. The coefficient of $\D_{12}$ on the right hand side above is $1 - \sum c_{ax} - \sum c_{yz} \geq 0$, so at most one of the coefficients can be non-zero.

We conclude the proof by determining which choices of $a,x,y$, and $z$ give effective representatives. If some $c_{ax}$ equals one, then $\D_{ax}$ appears with coefficient $-1$ if $x \in J$, and coefficient zero otherwise (recall that $J \subseteq \{1,2,n\}^c$). Likewise, if one of the $c_{yz}$ equals one, then the coefficient of $\D_{yz}$ is non-negative if and only if $\{y,z\} \cap J = \emptyset$. If $x \notin J$ and $c_{ax} = 1$, it is easy to see that $d = h_{ax} - \sum_{T \subseteq J} \D_{J \cup \{n\}} \in \orthNb$, while for $y,z \notin J$ with $c_{yz} = 1$, we obtain $d = h_{yz} - \sum_{T \subseteq J} \D_{J \cup \{n\}} \in \orthNb$.

In summary, the effective integral representatives for $[F_{J,n}]$ are in bijection with choices $a,b \in (J \cup \{n\})^c$, hence there are exactly $\binom{n-|J| - 1}{2}$ such representatives.
\end{proof}

To conclude that there are relations in the degree $[F_{J,n}]$ part of $\cox(\M_{0,n})$, the final ingredient is to calculate $h^0(\M_{0,n}, [F_{J,n}])$.
\begin{lemma}
\label{lemma:h0F}
The dimension of $H^0(\M_{0,n}, [F_{J,n}])$ is $n - |J| - 2$.
\end{lemma}
\begin{proof}
The proof follows by showing that $\phi_{J,n} = t_n \circ \pi_J: \M_{0,n} \to \PP^{n - |J| - 3}$ is an algebraic fiber space, followed by an application of the projection formula. The algebraic fiber space property requires that $\phi_{J,n}$ be a surjective projective morphism between reduced, irreducible varieties with connected fibers (this last condition is equivalent to the usual requirement for $f:X \to Y$ that $f_*(\cO_X) = \cO_Y$ since projective space is normal---see Definition 2.1.11 of \cite{MR2095471}).

The composition of blow-ups $t_n: \M_{0, n - |J|} \to \PP^{n - |J| - 3}$ clearly satisfies the necessary properties. Forgetful morphisms are proper and surjective with connected fibers, since each is a composition of forgetful morphisms forgetting a single point, while in the case of a single point, the morphism $\pi_{j}: \M_{0,n} \to \M_{0, n-1}$ is in fact the universal curve over $\M_{0,n-1}$ (\cite{MR702953}). Since the composition $\phi_{J,n}$ is projective, it follows that $\phi_{J,n}$ determines an algebraic fiber space. The rest of the proof now proceeds as in Lemma \ref{lemmaCox:pbH0}.
\end{proof}

\begin{cor}
The $[F_{J,m}]$-graded parts $\cox(\M_{0,n})$ have relations.
\end{cor}
\begin{proof}
By Proposition \ref{prop:effFJn} and Lemma \ref{lemma:h0F}, there are
\begin{gather*}
\binom{n - |J| - 1}{2} - (n - |J| - 2) = \binom{n - |J| - 2}{2}
\end{gather*}
linearly independent relations in the degree $[F_{J,m}]$. By assumption, $|J| \leq n-4$, so this number is positive.
\end{proof}

Since generators of $\cox(\M_{0,n})$ have not been determined for $n \geq 7$, it is possible that the classes $[F_{J,n}]$ can be expressed as non-negative sums of other effective divisors, which would increase the number of generators of the $[F_{J,n}]$-graded part of $\cox(\M_{0,n})$. In the next lemma, however, we show that this is not the case for Keel-Vermeire divisors, which are effective extremal divisors on $\M_{0,n}$ ($n \geq 6$) whose classes cannot be represented as effective sums of boundary divisors. While there are known to be additional extremal divisors in $\M_{0,n}$ (see \cite{ct}), their classes have not yet been calculated.

With respect to the Kapranov blow-up construction, the Keel-Vermeire divisor $Q_{(12)(34)}$ of $\M_{0,n}$ is the proper transform of the unique quadric surface in $\PP^3$ passing through the points $l_1, \ldots, l_5$ and containing the lines $l_{13}, l_{14}, l_{23}, l_{24}$ (recall notational convention \ref{notation:pt}). The remaining fourteen Keel-Vermeire divisors result from permuting the indices $1, \ldots, 5$ (note that $Q_{(ab)(cd)} = Q_{(cd)(ab)}$). We refer to \cite{MR1882122} for further discussion.

The class of $Q_{(12)(34)}$ in $\Pic(\M_{0,n})_\QQ$ is
\begin{gather*}
[Q_{(12)(34)}] = 2 H - \sum_{i=1}^5 [E_i] - \sum_{\substack{a \in \{1,2\}, \, b \in \{3,4\}}} [E_{ab}].
\end{gather*}
For higher $n$, the Keel-Vermeire divisors are pull-backs of those on $\M_{0,6}$. 

\begin{lemma}
\label{lemma:noKV}
The classes $[F_{J,n}]$ cannot be represented as an effective sum involving Keel-Vermeire divisors.
\end{lemma}
\begin{proof}
Let $l$ be the proper transform of a generic line in $\PP^{n-3}$ in $\M_{0,n}$. Then by the projection formula, $l$ intersects all exceptional divisors trivially, while $l \cdot H = 1$, so $l \cdot [F_{J,n}] = 1$. Since $l$ intersects all classes of effective divisors non-negatively, the intersection of $l$ with any divisor class represented by a Keel-Vermeire plus other effective divisors is greater than or equal to two, hence no such representative is possible.
\end{proof}

Note that the claim of Lemma \ref{lemma:noKV} also holds for any generator of $\cox(\M_{0,n})$ having intersection number greater than one with a generic line. For the remainder, we consider only $n = 6$, though proofs will be given that would apply for higher $n$ if an analog of Lemma \ref{lemma:noKV} were to apply for remaining generators.

Denote generators of $\cox(\M_{0,6})$ by
\begin{gather*}
\cG = \{x_J, y_{(ab)(cd)}: J \textrm{ boundary index}; a,b,c,d \in \{1, \ldots, 5\}\},
\end{gather*}
and consider the map
\begin{gather*}
\CC[\cG]_{[F_{J,m}]} \to \cox(\M_{0,6})_{[F_{J,m}]},
\end{gather*}
where the subscript denotes the $[F_{J,m}]$-graded parts of these rings, and the kernel is denoted by $I$. We conclude by showing that all relations in the $[F_{J,m}]$ parts of $I$ are generated by Pl\"ucker relations.
\begin{defn}
Let $I_{pl} \subseteq I$ be the relations generated in degrees $F_{J,m}$ with $|J| = n-4$, that is, corresponding to pull-backs of $\cO_{\PP^1}(1)$ under all forgetful morphisms $\pi_J: \M_{0,n} \to \M_{0,4} \cong \PP^1$. We call elements of $I_{pl}$ \emph{Pl\"ucker relations}, and label them $p_I$. 

Explictly, the Pl\"ucker relation for $I = \{i,j,k,l\}$ is
\begin{gather}
\label{eq:PlueckerRels}
p_I = \prod_{\substack{i,j \in T , \\ k,l \in T^c}} x_T - \prod_{\substack{i,k \in T , \\ j,l \in T^c}} x_T + \prod_{\substack{i,l \in T , \\ j,k \in T^c}}x_T 
\end{gather}

\noindent Each monomial in the $p_I$ is called a \emph{Pl\"ucker monomial}.
\end{defn}

\begin{prop} 
\label{prop:Pluecker}
The Pl\"ucker relations generate all relations in the $[F_{J,m}]$-degrees of $\cox(\M_{0,6})$.
\end{prop}
\begin{proof}
Since $|J| \leq n-4 = 2$, by the action of the symmetric group, we can assume that the index sets of `forgotten points' $J$ are contained in $\{1, \ldots, 4\}$, and hence we will write $F_J$ instead of $F_{J,m}$.

Consider a forgetful index $J \subseteq\{1, \ldots, 4\}$. As in the proof of Proposition \ref{prop:effFJn}, all monomials in degree $[F_J]$ are given by picking $a,b \in \{1, \ldots, 5\} \setminus J$ and taking in Equation (\ref{eq:FIKap}) the corresponding representative $H_{ab}$ of the hyperplane class as given by Equation (\ref{eq:dictionaryDab}). Hence the monomials in $\CC[\cG]_{[F_J]}$ have the form
\begin{gather*}
m_J(a,b) = \big( \prod_{\substack{a,b \in T^c , \\ 6 \in T}} x_T \big)/ \big(\prod_{\substack{ 6 \in T\\ J \cup \{6\} \supseteq T}} x_T\big) = \prod_{\substack{a,b \in T^c, 6 \in T , \\ J \cup \{6\} \nsupseteq T}} x_T.
\end{gather*}

Let $k \in (\{a,b,6\} \cup J)^c$. As above, we can write the monomial $m_J(a,b)$ as
\begin{gather}
\label{eq:mJab}
m_J(a,b) = \big(\prod_{\substack{a,b,k \in T^c, \\ 6 \in T, \\ J \cup \{6\} \nsupseteq T}} x_T \big) \big(\prod_{\substack{a,b \in T^c, \\k, 6 \in T}} x_T \big),
\end{gather}
where we have omitted the condition $J \nsubseteq T$ in the second term, since $k \in T$ implies that $T \nsubseteq J \cup \{6\}$. Note that the second product on the right hand side of Equation (\ref{eq:mJab}) is a Pl\"ucker monomial.

We next show that each $[F_J]$ is a pull-back from a permutohedral space $\LM_{4}(i,j)$. Recall that we have assumed $J \subseteq \{1, \ldots, 4\}$, and by $\LM_{4}$ we mean that the points labeled by $5$ and $6$ are taken to be the poles. The claim now follows immediately, since each $J' \subsetneq J$ also labels an exceptional divisor of $\LM_{4}$ under the composition of blow-ups $f: \M_{0,6} \to \LM_{4}$, and so by Corollary \ref{cor:pbFormula}, if we define
\begin{gather*}
F'_J = H' - \sum_{T \subsetneq J} E'_{T},
\end{gather*}
 then $[F_J] = f^*[F'_J]$.


To conclude, we consider the following commutative diagram:
\begin{displaymath} 
\xymatrix{
& 
(\CC[\cG]/I_{pl})_{[F_{J}]}
\ar[dr]^{g}
&
\\
(\cox(\LM_{4}))_{[F'_{J}]}
\ar[ur]^{i}
\ar[rr]_{\cong}^{f^*}
&&
(\cox(\M_{0,6}))_{[F_{J}]}
}
\end{displaymath}
where $i$ is the map $x_T \mapsto x_T + I_{pl}$, and $g$ is the surjection
\begin{gather*}
(\CC[\cG]/I_{pl})_{[F_J]} \twoheadrightarrow (\CC[\cG]/I)_{[F_J]} \cong (\cox(\M_{0,6}))_{[F_J]}.
\end{gather*}
The proposition will be proven once we show that $i$ is also surjective.

Let $m_J(a,b)$ be a monomial in $\CC[\cG]_{[F_J]}$ as in Equation (\ref{eq:mJab}). Note that all terms of the monomial $m_J(a,5)$ are already in the image of $i$, so we may assume that $5 \notin \{a,b\}$. We may therefore choose $k$ of Equation (\ref{eq:mJab}) to be $5$. Modulu $I_{pl}$, we may rewrite the second product using the Pl\"ucker relation $(a,b,5,6)$ to give
\begin{gather*}
m_J(a,b) \equiv \big(\prod_{\substack{a,b,5 \in T^c, \\ 6 \in T, \\ J \cup \{5\} \nsupseteq T}} x_T \big) \big(\prod_{\substack{a, 5 \in T^c, \\b, 6 \in T}} x_T - \prod_{\substack{a,6 \in T^c, \\b, 5 \in T}} x_T \big) + I_{pl},
\end{gather*}
which is in the image of $i$, since for each term $x_T$, $T$ either contains $5$ or $6$, but never both or neither.
\end{proof}

It is not hard to prove via methods used in this section that relations also exist in certain degrees involving Keel-Vermeire classes arising from relations described in Example 6.6 of \cite{hilbChow}. In future work, we plan to link results of this paper with the graded Betti numbers approach of \cite{MR2529093} to determine if all relations in $\cox(\M_{0,6})$ reside in these degrees, Pl\"ucker degrees, or the degrees determined in \cite{hilbChow}.

\bibliographystyle{alpha}

\bibliography{../../../bibliography.bib}

\end{document}